\documentclass[11pt]{amsart}
\usepackage{geometry}                
\usepackage{graphicx}
\usepackage{amssymb}
\usepackage{epstopdf}
\title{The Rado Simplicial Complex}
\author{Michael Farber, Lewis Mead and Lewin Strauss}
\date{28 January, 2020}  
\address{School of Mathematical Sciences \\
Queen Mary University of London\\
London E1 4NS\\
United Kingdom}
\email{m.farber@qmul.ac.uk; lewis.mead@qmul.ac.uk; l.strauss@qmul.ac.uk}
\thanks{Michael Farber was partially supported by a grant from the Leverhulme Foundation.}   

\usepackage{amsmath}
\usepackage{amsfonts}
\usepackage{amssymb}
\usepackage{amsthm}
\usepackage{color}
\usepackage[parfill]{parskip}
\usepackage{etoolbox}


\newtheorem{theorem}{Theorem}
\newtheorem{corollary}[theorem]{Corollary}
\newtheorem{lemma}[theorem]{Lemma}
\newtheorem{prop}[theorem]{Proposition}

\newtheorem{definition}[theorem]{Definition}

\newtheorem{example}[theorem]{Example}
\newtheorem{remark}[theorem]{Remark}

\newcommand{\lk}{{\rm {Lk}}}

\newcommand{\Z}{{\mathbb Z}}
\newcommand{\N}{{\mathbb N}}

\newcommand{\E}{{\Bbb E}}

\newcommand{\R}{\Bbb R}

\begin{document}
\maketitle
\begin{abstract}
A Rado simplicial complex $X$ is a generalisation of the well-known Rado graph \cite{C}. $X$ is a countable simplicial complex which contains any countable simplicial complex as its induced subcomplex. The Rado simplicial complex is highly symmetric, 
it is homogeneous: 
any isomorphism between finite induced subcomplexes can be extended to an isomorphism of the whole complex. We show that the Rado complex $X$ is unique up to isomorphism and suggest several explicit constructions. We also show that a random simplicial complex on countably many vertices is a Rado complex with probability 1. The geometric realisation $|X|$ of a Rado complex is contractible and is homeomorphic to an infinite dimensional simplex. We also prove several other interesting properties of the Rado complex $X$, for example we show that removing any finite set of simplexes of $X$ gives a complex isomorphic to $X$. 
\end{abstract}
\section{Introduction}

The Rado graph $\Gamma$ is a remarkable combinatorial object, it can be characterised by its universality and homogeneity. 
The graph $\Gamma$ is universal in the sense that any graph with finitely or countably many vertices is isomorphic to an induced subgraph of $\Gamma$. Besides,  any isomorphism of between finite induced subgraphs of $\Gamma$ can be extended to the whole $\Gamma$ (homogeneity). $\Gamma$ is unique up to isomorphism and there are many different ways of constructing it explicitly. 
Erdős and Rényi \cite{ER} showed  that an infinite random graph is isomorphic to $\Gamma$ with probability 1. 
The Rado graph $\Gamma$ was introduced by Richard Rado \cite{Ra} in 1964, see also previous construction of W. Ackermann \cite{Ack}.
One may mention surprising robustness of $\Gamma$: removing any finite set of vertices and edges produces a graph isomorphic to $\Gamma$. 
Many other interesting results about the Rado graph $\Gamma$ are known, we refer the reader to the survey \cite{C}. 

The paper \cite{BTT} studies the $\rm {Fra \ddot{\imath} ss\acute e}$ limit of the class of all finite simplicial complexes. The authors show that the geometric realisation of the limit complex is homeomorphic to the infinite dimensional simplex. 

The goal of the present paper is to describe a {\it Rado simplicial complex} $X$. It is a simplicial complex with countably many vertices which is universal and homogeneous; we show that these two properties characterise $X$ up to isomorphism. R. Rado mentions universal simplicial complexes in his paper \cite{Ra} and gives a specific construction. The 1-dimensional skeleton of $X$ is a Rado graph. 
We show that $X$ can also be characterised by its {\it ampleness} which is a high-dimensional generalisation of the well-known extension property of the Rado graph. We also prove that $X$ is {\it robust}, i.e. removing any set of finitely many simplexes produces a simplicial complex isomorphic to $X$. If the set of vertices of $X$ is partitioned into finitely many parts, the simplicial complex induced on at least one of these parts is isomorphic to $X$. The link of any simplex of $X$ is a Rado complex. 
We prove that a random infinite simplcial complex is isomorphic to $X$ with probability 1. 
Besides, we show that the geometric realisation $|X|$ of the Rado complex is homeomorphic to the infinite dimensional simplex.

\section{The Definition of the Rado complex}

\subsection{Basic terminology} {\it A simplicial complex} $X$ is a set of vertexes $V(X)$ and a set of non-empty finite subsets of $V(X)$, called {\it simplexes} such that any vertex $v\in V(X)$ is a simplex $\{v\}$ and any subset of a simplex is a simplex. 
A simplicial complex $X$ is said to be countable (finite) if its vertex set 
$V(X)$ is countable (finite). The symbol $F(X)$ stands for the set of all simplexes of $X$. 
For a simplex $\sigma\in F(X)$ we shall also write $\sigma\subset X$. 

Two simplicial complexes are {\it isomorphic} if there is a bijection between their vertex sets which induces a bijection between the sets of simplexes.

The standard simplex $\Delta_n$ has the set of vertexes $\{1, 2, \dots, n\}$ with all non-empty subsets as simplexes. Another standard simplex is $\Delta_\N$; its vertex set is $\N=\{1, 2, \dots\}$ and all non-empty finite subsets of $\N$ are simplexes.

A simplicial subcomplex $Y\subset X$ is said to be {\it induced} if any simplex $\sigma\in F(X)$ with all its faces in $V(Y)$ belongs to $F(Y)$. The induced subcomplex $Y\subset X$ is completely determined by the set of its vertices, $V(Y)\subset V(X)$. We shall use the notation $Y=X_U$ where $U=V(Y)$. 

%


For a vertex $v\in V(X)$ the symbol $\lk_X(v)$ stands for {\it the link} of $v$; the latter is the union of simplexes $\sigma$ of $X$ with
$v\notin \sigma$ and $v\sigma\subset X$. 


\begin{definition}\label{def1}
(1) A countable simplicial complex $X$ is said to be {\it universal} if any countable simplicial complex is isomorphic to an induced subcomplex of $X$. (2) We say that $X$ is {\it homogeneous} if for any two finite induced subcomplexes $X_U, X_{U'}\subset X$ and for any isomorphism $f:X_U\to X_{U'}$ there exists an isomorphism $F:X\to X$ with $F|X_U=f$. (3) A countable simplicial complex $X$ is a Rado complex if it is universal and homogeneous. 
\end{definition}

It is clear that the 1-skeleton of a Rado complex is a Rado graph; the latter can be defined as a universal and homogeneous graph having countably many vertexes, see \cite{C}.

We prove in this paper: 
\begin{theorem}
Rado simplicial complexes exist and any two Rado complexes are isomorphic. 
\end{theorem}

The following property is a useful criterion of being a Rado complex: 
\begin{definition}\label{rado}
We shall say that a countable simplicial complex $X$ is ample if for any finite subset $U\subset V(X)$ and for any simplicial subcomplex 
$A\subset X_U$ there exists a vertex $v\in V(X)-U$ such that 
\begin{eqnarray}\label{link}
\lk_X(v)\cap X_U=A.\end{eqnarray}
\end{definition}
\begin{remark}\label{rk1}{\rm 
Condition (\ref{link}) can equivalently be expressed as 
\begin{eqnarray}\label{cone}
X_{U'} = X_U\cup (vA),
\end{eqnarray}
where $U'=U\cup \{v\}$
and $vA$ denotes the cone with apex $v$ and base $A$. In literature the cone $vA$ is also sometimes denoted 
$v\ast A$, the simplicial join of a vertex $v$ and complex $A$. 
}
\end{remark}
\begin{remark}\label{rm2} {\rm 
Suppose that $X$ is a simplicial complex with countable set of vertexes $V(X)$. One may naturally consider exhaustions 
$U_0\subset U_1\subset U_2\subset \dots\subset V(X)$ consisting of finite subsets $U_n$ satisfying $\cup U_n = V(X)$. In order to check  that $X$ is ample as defined in Definition \ref{rado} it is sufficient to verify that for any $n\ge 0$ and for any subcomplex $A\subset X_{U_n}$ there exists
a vertex $v\in V(X)-U_n$ satisfying $\lk_X(v)\cap X_{U_n}= A$. 
}
\end{remark}

\begin{remark} {\rm Suppose that $X$ is an ample simplicial complex.
Given finitely many distinct vertexes 
$u_1, \dots, u_m, v_1, \dots, v_n\in V(X)$, there exists a vertex $z\in V(X)$ which is adjacent to $u_1, \dots, u_m$ and nonadjacent to $v_1, \dots, v_n$. To see this we apply Definition \ref{rado} with $U= \{u_1, \dots, u_m, v_1, \dots, v_n\}$ and $A=\{u_1, \dots, u_m\}$. 
This shows that the 1-skeleton of a Rado complex satisfies the defining property of the Rado graph \cite{C}. This also shows that ampleness is a high dimensional generalizaton of this graph property.
}
\end{remark}

The following property of ample complexes will be useful in the sequel. 

\begin{lemma}\label{lm16} Let $X$ be an ample complex and let $L'\subset L$ be a pair consisting of a finite simplical complex $L$ and an  induced subcomplex $L'$. Let $f': L' \to X_{U'}$ be an isomorphism of simplicial complexes, where $U'\subset V(X)$ is a finite subset. Then there exists a finite subset $U\subset V(X)$ containing $U'$ and an isomorphism $f: L \to X_{U}$ with $f|L' =f'$. 
\end{lemma}

\begin{proof}
It is enough to prove this statement under an additional assumption that $L$ has a single extra vertex, i.e. $V(L) - V(L') =1$. In this case 
$L$ is obtained from $L'$ by attaching a cone $wA$ where $w\in V(L) - V(L')$ denotes the new vertex and $A\subset L'$ is a subcomplex (the base of the cone). Applying the defining property of the ample complex to the subset $U'\subset V(X)$ and the subcomplex $f'(A)\subset X_{U'}$
we find a vertex $v\in V(X)-U'$ such that $\lk_X(v)\cap X_{U'} =f(A)$. We can set $U=U'\cup \{v\}$ and extend $f'$ to the isomorphism 
$f: L\to X_{U}$ by setting $f(w)=v$.
\end{proof}

\begin{theorem} \label{thm1} A simplicial complex is Rado if and only if it is ample. 
\end{theorem}
\begin{proof} Suppose $X$ is a Rado complex, i.e. $X$ is universal and homogeneous. Let $U\subset V(X)$ be a finite subset and let $A\subset X_U$ be a subcomplex of the induced complex. Consider an abstract simplicial complex $L=X_U\cup wA$ which obtained from $X_U$ by adding a cone $wA$ with vertex $w$ and base $A$ where $X_U\cap wA= A$. Clearly, $V(L) = U\cup\{w\}$. By universality, we may find a subset $U'\subset V(X)$ and an isomorphism $g:L \to X_{U'}$. Denoting $w_1=g(w)$, $A_1=g(A)$ and $U_1=g(U)$ we have 
$X_{U'} = X_{U_1} \cup w_1 A_1.$
Obviously, $g$ restricts to an isomorphism $g| X_U\, :\,  X_U \to X_{U_1}$. 
By the homogeneity property we can find an isomorphism $F:X\to X$ with $F|X_U =g|X_U$. Denoting $v=F^{-1}(w_1)$ we shall have 
$X_{U\cup \{v\}} = X_U \cup vA$ as required, see Remark \ref{rk1}. 

Now suppose that $X$ is ample. To show that it is universal consider a simplicial complex $L$ with at most countable set of vertexes $V(L)$. 
We may find a chain of induced subcomplexes $L_1\subset L_2\subset \dots$ with $\cup L_n=L$ and each complex $L_{n}$ has exactly $n$ vertexes. Then $L_{n+1}$ obtained from $L_n$ by adding a cone $v_{n+1}A_n$ where $v_{n+1}$ is the new vertex and $A_n\subset L_n$ is a simplicial subcomplex. We argue by induction that we can find a chain of subsets $U_1\subset U_2\subset \dots\subset V(X)$ 
and isomorphisms $f_n: L_n \to X_{U_n}$ satisfying $f_{n+1}|L_n = f_n$. If $U_n$ and $f_n$ are already found then the next set $U_{n+1}$ and the isomorphism $f_{n+1}$ exist because $X$ is ample: we apply Definition \ref{rado} with $U=U_n$ and $A=f_n(A_n)$ and we set $U_{n+1} = U_n \cup \{v\}$ where $v$ is the vertex given by Definition \ref{rado}. The sequence of maps $f_n$ defines an injective map $f: V(L)\to V(X)$ and produces an isomorphism between $L$ and the induced subcomplex $X_{f(V(L))}$. 

The fact that any ample complex is homogeneous follows from Lemma \ref{lmis} below. We state it in a slightly more general form so that it also implies the uniqueness of Rado complexes. 
\end{proof}

\begin{lemma}\label{lmis}
Let $X$ and $X'$ be two ample complexes and let $L\subset X$ and $L'\subset X'$ be two induced finite subcomplexes. Then any isomorphism $f: L\to L'$ can be extended to an isomorphism $F: X\to X'$. 
\end{lemma} 
\begin{proof}
We shall construct chains of subsets of the sets of vertexes $U_0\subset U_1\subset \dots\subset V(X)$ and $U'_0\subset U'_1\subset \dots\subset V(X')$ such that $\cup U_n=V(X)$,  $\cup U'_n=V(X')$, $X_{U_0}=L$, $X_{U'_0}=L'$, 
and $|U_{n+1}-U_n|=1$, $|U'_{n+1}-U'_n|=1$. We shall also construct isomorphisms $f_n: X_{U_n} \to X_{U'_n}$ satisfying $f_0=f$ and $f_{n+1}|X_{U_n} =f_n$. The whole collection $\{f_n\}$ will then define a required isomorphism $F: X\to X'$ with $F|L=f$. 

To constructs these objects we shall use the well known back-and-forth procedure. Enumerate vertices $V(X)-V(L) =\{v_1, v_2, \dots\}$ and 
$V(X')-V(L') =\{v'_1, v'_2, \dots\}$ and start by setting $U_0=V(L)$, $U'_0=L'$ and $f_0=f$. We act by induction and describe $U_n$, $U'_n$ and $f_n$ assuming that the objects $U_i$, $U'_i$ and $f_i:U_i\to U'_i$ have been already defined for all $i<n$. 

The procedure will depend on the parity of $n$. For $n$ odd we find the smallest $j$ with $v_j\notin U_{n-1}$ and set $U_n=U_{n-1}\cup \{v_j\}$. 
Applying Lemma \ref{lm16} to the simplicial complexes $L=X_{U_n}$, $L'=X_{U_{n-1}}$ and the isomorphism $f_{n-1}: X_{U_{n-1}}\to X'_{U'_{n-1}}$ we obtain a subset $U'_n\subset V(X')$ containing $U'_{n-1}$ and an isomorphism $f_n: X_{U_n}\to X'_{U'_n}$ extending $f_{n-1}.$

For $n$ even we proceed in the reverse direction. We find the smallest $j$ with $v'_j\notin U'_{n-1}$ and set $U'_n=U'_{n-1}\cup \{v'_j\}$. 
Next we applying Lemma \ref{lm16} to the simplicial complexes $L=X'_{U'_n}$, $L'=X'_{U'_{n-1}}$ and the isomorphism $f_{n-1}^{-1}: X'_{U'_{n-1}}\to X_{U_{n-1}}$. We obtain a subset $U_n\subset V(X)$ containing $U_{n-1}$ and an isomorphism $f_n^{-1}: X'_{U'_n}\to X_{U_n}$ extending $f_{n-1}^{-1}.$
\end{proof}

\begin{corollary}\label{unique}
Any two Rado complexes are isomorphic. 
\end{corollary}
\begin{proof}
This follows from Theorem \ref{thm1} with subsequent applying Lemma \ref{lmis} with $L=L'=\emptyset$. 
\end{proof}

\begin{remark}{\rm 
In Definition \ref{def1} we defined universality with respect to arbitrary countable simplicial subcomplexes. 
A potentially more restrictive definition 
dealing only with finite subcomplexes together with homogeneity is in fact equivalent to Definition \ref{def1}; this follows from the arguments used in the proof of Theorem \ref{thm1}. 
}
\end{remark}
\begin{remark}{\rm There is a variation of the whole study presented in this paper when one is interested in simplicial complexes containing no induced subcomplexes isomorphic to 
$\partial \Delta_d$, for a fixed $d>0$. Note that $\partial \Delta_d$ is the simplicial complex with the vertex set $\{1, 2, \dots, d\}$ and with all non-empty proper subsets of $\{1, 2, \dots, d\}$ as simplexes. Simplicial complexes with this property 
can also be characterised as having no external $d$-dimensional simplexes. 

Definition \ref{rado} can be modified for this case: A simplicial complex $X$ is $d$-{\it ample} if for any finite subset $U\subset V(X)$ and for any simplicial subcomplex $A\subset X_U$ satisfying $E_{d-1}(A|X_U)=\emptyset$ there exists a vertex $v\in V(X)-U$ such that $\lk_X(v)\cap X_U = A$. Here the symbol $E_{d-1}(A|X_U)$ stands for the set of all $(d-1)$-dimensional simplexes $\tau\in X_U$ with $\partial \tau\subset A$ but $\tau\notin A$. 

It can be shown (similarly to the above) that $d$-ample simplicial complexes exist and are universal and homogeneous with respect to the class of complexes with no induced $\partial \Delta_d$. 

Moreover Theorem \ref{thm26} can also be generalised: random infinite simplicial complexes are universal  with probability one for this specific class (see \S 6 and \S 7)
when the probability parameters $p_\sigma$ satisfy 
$p_\sigma=1$ for all simplexes $\sigma$ with $\dim \sigma =d$ and $0<p_-\le p_\sigma\le p_+<1$ for $\dim \sigma \not= d$. 

}
\end{remark}

\section{Deterministic constructions of Rado complexes}

\subsection{An inductive construction} One may construct a Rado simplicial complex $X$ inductively as the union of a chain of finite induced simplicial subcomplexes $$X_0\subset X_1\subset X_2\subset \dots, \quad \cup_{n\ge 0}X_n \, =\,  X.$$ 
Here $X_0$ is a single point and each complex $X_{n+1}$ is obtained from 
$X_n$ by first adding a finite set of vertices $v(A)$, labeled by subcomplexes 
$A\subset X_n$ (including the case when $A=\emptyset$); secondly, we construct the cone $v(A) \ast A$ with apex $v(A)$ and 
base $A$, and thirdly we attach each such cone $v(A)\ast A$ to $X_n$ along the base $A\subset X_n$. 
%
%
%
%
Thus, \begin{eqnarray}
X_{n+1} = X_n \, \cup\,  \bigcup_{A} (v(A)\ast A). 
\end{eqnarray}

To show that the complex $X=\cup_{n\ge 0} X_n$ is ample, i.e. a Rado complex, we refer to Remark \ref{rm2} and observe that 
any 
subcomplex $A\subset X_n$ the vertex $v=v(A)\in V(X_{n+1})$ 
satisfies $\lk_X(v)\cap X_n =A.$

\subsection{An explicit construction} Here we shall give an explicit construction of a Rado complex $X$. To describe it we shall use the sequence $\{p_1, p_2, \dots, \}$ of all primes in increasing order, where $p_1=2$, $p_2=3$, etc. 

The set of vertexes $V(X)$ is the set of all positive integers $\N$. Each simplex of $X$ is uniquely represented by an increasing sequence $a_0<a_1< \dots <a_k$ with certain properties. Subsequences of $a_0<a_1< \dots <a_k$ are obtained by erasing one or more elements in the sequence. 
%
%

\begin{definition} 
(1) A sequence $a_0<a_1$ is a 1-dimensional simplex of $X$ if and only if $p_{a_0}$-th binary digit of $a_1$ is 1. 
(2) We shall say that an increasing sequence of positive integers $0<a_0<a_1< \dots< a_k$ represents a simplex of $X$ if 
all its proper subsequences are in $X$ and additionally the $p_{a_0}p_{a_1} \dots p_{a_{k-1}}$-th binary digit of $a_k$ is 1. 
\end{definition}
\begin{prop}
The obtained simplicial complex $X$ is Rado. 
\end{prop}
\begin{proof}

With any increasing sequence $\sigma$ of positive integers $0<a_0<a_1< \dots< a_k$ we associate the product 
$$N_\sigma = p_{a_0}p_{a_1}\dots p_{a_k},$$
which is an integer without multiple prime factors. Note that for two such increasing sequences $\sigma$ and $\sigma'$ one has $N_\sigma=N_{\sigma'}$ if and only if $\sigma$ is identical to $\sigma'$. 

Given a finite subset $U\subset V(X)$ and a simplicial subcomplex $A\subset X_U$, consider the vertex 
\begin{eqnarray}\label{nsigma}
v = \sum_{\sigma\in F(A)} 2^{N_\sigma} + 2^{K_U}\, \in\,  V(X)
\end{eqnarray}
where $$K_U = 1 + \prod_{w\in U} p_w.$$
The binary expansion of $v$ has ones exactly on positions $N_\sigma$ where $\sigma\in F(A)$ and it has zeros on all other positions except an additional 1 at position $K_U$. Note that $K_U>N_\sigma$ for any simplex $\sigma\subset X_U$. In particular, we see that vertex $v$ defined by (\ref{nsigma}) satisfies $v> w$ for any $w\in U$. 

Consider a simplex $\sigma\subset X_U$. By definition, the simplex $v\sigma$ with apex $v$ and base $\sigma$ lies in $X$ if and only if the 
$N_\sigma$-th binary digit of $v$ is 1. We see from (\ref{nsigma}) that it happens if and only if $\sigma\subset A$. This means that 
$\lk_X(v)\cap X_U=A$ and hence the complex $X$ is a Rado complex. 
\end{proof}
\section{Some properties of the Rado complex}

\begin{lemma}\label{lm31}
Let $X$ be a Rado complex, let $U\subset V(X)$ be a finite set and let $A\subset X_U$ be a subcomplex. 
Let $Z_{U,A}\subset V(X)$ denote the set of vertexes $v\in V(X)-U$ satisfying (\ref{link}). Then $Z_{U,A}$ is infinite and the induced complex on $Z_{U,A}$ is also a Rado complex. 
\end{lemma}
\begin{proof} Consider a finite set $\{v_1, \dots, v_N\}\subset Z_{U, A}$ of such vertexes. 
One may apply Definition \ref{rado} to the set 
$U_1=U\cup \{v_1, \dots, v_N\}$ and to the subcomplex $A\subset X_{U_1}$ to find another vertex $v_{N+1}$ satisfying the condition of Definition \ref{rado}. This shows that $Z_{U,A}$ must be infinite. 

Let $Y\subset X$ denote the subcomplex induced by $Z_{U, A}$. 
Consider a finite subset $U'\subset Z_{U,A}=V(Y)$ and a subcomplex $A'\subset X_{U'}=Y_{U'}$. Applying the condition of Definition \ref{rado} 
to the set $W= U\cup U'\subset V(X)$ and to the subcomplex $A\sqcup A'$ we find a vertex $z\in V(X)-W$ such that 
\begin{eqnarray}\label{links2}
\lk_X(z) \cap X_W = A\cup A'.\end{eqnarray} 
Since $X_{W} \supset X_U\cup X_{U'}$, the equation (\ref{links2}) implies $\lk_X(z)\cap X_U=A$, i.e. $z\in Z_{U, A}$. 
Intersection both sides of (\ref{links2}) with $X_{U'}=Y_{U'}$ and using $\lk_Y(z) =\lk_X(z)\cap Y$ (since $Y$ is an induced subcomplex) we obtain
$$\lk_Y(z) \cap Y_{U'}= A'$$
implying that $Y$ is Rado. 
\end{proof}

\begin{corollary}
Let $X$ be a Rado complex and let $Y$ be obtained from from $X$ by selecting a finite number of simplexes $F\subset F(X)$ 
and deleting all simplexes $\sigma\in F(X)$ which contain simplexes from $F$ as their faces. Then $Y$ is also a Rado complex. 
\end{corollary}
\begin{proof}
Let $U\subset V(Y)$ be a finite subset and let $A\subset Y_U$ be a subcomplex. We may also view $U$ as a subset of $V(X)$ and then 
$A$ becomes a subcomplex of $X_U$ since $ Y_U\subset X_U$. The set of vertexes $v\in V(X)$ satisfying $\lk_X(v)\cap X_U=A$ is infinite (by Lemma \ref{lm31}) and thus we may find a vertex $v\in V(X)$ which is not incident to simplexes from the family $F$. Then $\lk_Y(v)=\lk_X(v)\cap Y$ and we obtain 
 $\lk_Y(v)\cap Y_U=A$. \end{proof}

\begin{corollary}\label{lm166} Let $X$ be a Rado complex. If the vertex set $V(X)$ is partitioned into a finite number of parts then the induced subcomplex on at least one of these parts is a Rado complex. 
\end{corollary}
\begin{proof} It is enough to prove the statement for partitions into two parts. 
Let $V(X)= V_1\sqcup V_2$ be a partition; denote by $X^1$ and $X^2$ the subcomplexes induced by $X$ on $V_1$ and $V_2$ correspondingly. 
Suppose that none of the  subcomplexes $X^{1}$ and $X^{2}$ is Rado. 
Then for each $i=1, 2$ there exists a finite subset $U_i\subset V_i$ and a subcomplex $A_i\subset X^i_{U_i}$ such that no vertex $v\in V_i$ satisfies 
$\lk_{X^i}(v) \cap X^i_{U_i}= A_i$. 
Consider the subset $U= U_1\sqcup U_2\subset V(X)$ and a subcomplex 
$A=A_1\sqcup A_2\subset X_{U}$. Since $X$ is Rado we may find a vertex $v\in V(X)$ with $\lk_X\cap X_U = A. $ Then $v$ lies in $V_1$ or $V_2$ and we obtain a contradiction, since $\lk_{X^i}(v) \cap X^i_{U_i} =A_i.$
\end{proof}

\begin{lemma}
In a Rado complex $X$, the the link of every simplex is a Rado complex. 
\end{lemma}
\begin{proof}
Let $Y=\lk_X(\sigma)$ be the link of a simplex $\sigma\in X$. To show that $Y$ is Rado, let $U\subset V(Y)$ be a subset and let $A\subset Y_U$ be a subcomplex. 
We may apply the defining property of the Rado complex to the subset $U'=U\cup V(\sigma)\subset V(X)$ and to the subcomplex 
$A\sqcup \bar\sigma\subset X_{U'}$; here $\bar \sigma$ denotes the subcomplex containing the simplex $\sigma$ and all its faces.
We obtain a vertex $w\in V(X)-U'$ with $\lk_X(w)\cap X_{U'} = A\sqcup \bar \sigma$ or equivalently, $X_{U'\cup w} = X_{U'} \cup wA$, see Remark \ref{rk1}. Note that $w\in Y=\lk_X(\sigma)$ since the simplex $w\sigma$ is in $X$. Besides, $Y_{U\cup w} =Y_U \cup wA$. Hence we see that the link $Y$ is also a Rado complex. 
\end{proof}
%
%
%
%
%
%
%
%
%

\section{Geometric realisation of the Rado complex}

Recall that for a simplicial complex $X$ {\it the geometric realisation} $|X|$ is the set of all functions $\alpha: V(X)\to [0,1]$ such that 
the support ${\rm {supp}}(\alpha)=\{v; \alpha(v)\not=0\}$ is a simplex of $X$ (and hence finite) and $\sum_{v\in X} \alpha(v)=1$, see \cite{Sp}. For a simplex $\sigma\in F(X)$ the symbol $|\sigma|$ denotes the set of all $\alpha\in |X|$ with ${\rm {supp}}(\alpha)\subset \sigma$. The set $|\sigma|$ has natural topology and is homeomorphic to the linear simplex lying in an Euclidean space. 
{\it The weak topology on} the geometric realisation $|X|$ has as open sets the subsets $U\subset |X|$ such that $U\cap |\sigma|$
is open in $ |\sigma|$ for any simplex $\sigma$.

\begin{lemma}
Let $X$ be a Rado complex. Then there exists a sequence of finite subsets $U_0\subset U_1\subset U_2\subset \dots\subset V(X)$ such that $\cup U_n = V(X)$ and for any $n=0, 1, 2, \dots$ the induced simplicial complex $X_{U_n}$ is isomorphic to a triangulation $L_n$ of the standard simplex $\Delta_{n+1}$ of dimension $n$.  Moreover, for any $n$ the complex $L_n$ is naturally an induced subcomplex of $L_{n+1}$ and the isomorphisms $f_n: X_{U_n} \to L_n$ satisfy $f_{n+1}|X_{U_n} = f_n$. 
\end{lemma}
\begin{proof} Let $V(X)=\{v_0, v_1, \dots\}$ be a labelling of the vertices of $X$. 
One constructs the subsets $U_n$ and complexes $L_n$ by induction stating from $U_0=\{v_0\}$ and $L_0=\{v_0\}$. Suppose that the sets $U_i$ and complexes $L_i$ with $i\le n$ have been constructed.  
Consider the subset $U'_{n+1} = U_n\cup \{v_i\}\subset V(X)$ where $i\ge 0$ is the smallest integer satisfying $v_i\notin U_n$. 
The induced simplicial complex $L'_{n+1}=X_{U'_{n+1}}$ has dimension $\le n+1$. Clearly, the complex $L'_{n+1}$ has the form 
$X_{U_n}\cup (v_i\ast A)$ for some subcomplex $A\subset X_{U_n}$. 
We shall apply Lemma \ref{lm16} to the abstract simplicial complexes 
$L'_{n+1}$ and $L_{n+1}$, where $L_{n+1}$ is a subdivision of the cone $v_i\ast X_{U_n}$ such that $L'_{n+1}$ is an induced subcomplex of $L_{n+1}$. Lemma \ref{lm16} gives a subset $U_{n+1}\subset V(X)$ containing $U'_{n+1}$ such that the induced complex $X_{U_{n+1}}$ is isomorphic to $L_{n+1}$. By construction $L_{n+1}$ is a triangulation of $v_i\ast X_{U_n}$, hence it is a triangulation of a simplex of dimension $n+1$.
Obviously, we have $\cup U_n = V(X)$. 
\end{proof}
\begin{theorem} The Rado complex is isomorphic to a triangulation of the simplex $\Delta_\N$. In particular, 
the geometric realisation $|X|$ of the Rado complex is homeomorphic to the infinite dimensional simplex $|\Delta_\N|$. 
\end{theorem}
\begin{proof}
It follows from the previous Lemma. 
\end{proof}

Note that the geometric realisation $|X|$ of a Rado complex $X$ (equipped with the weak topology)  does not satisfy the first axiom of countability and hence is not metrizable. This follows from the fact that $X$ is not locally finite. See \cite{Sp}, Theorem 3.2.8. 
\begin{corollary}
The geometric realisation $|X|$ of the Rado complex is contractible. 
\end{corollary}
\begin{proof} Corollary follows from the previous Theorem. We also give a short independent proof below. 
Let $X$ be a Rado complex. By the Whitehead theorem we need to show that any continuous map $f: S^n\to X$ is homotopic to the constant map. By the Simplicial Approximation theorem $f$ is homotopic to a simplicial map $g: S^n\to X$. The image $g(S^n)\subset X$ is a finite subcomplex. Applying the property of Definition \ref{rado} to the set of vertices $U$ of $g(S^n)$ and to the subcomplex $A=X_U$ we find a vertex $v\in V(X)-U$ such that the cone $vA$ is a subset of $X$. Since the cone is contractible, we obtain that $g$, which is equal the composition 
$S^n \to A\to vA \to X$, is null-homotopic.
\end{proof}
\begin{remark}{\rm 
The geometric realisation of a simplicial complex carries another natural topology, the metric topology, see \cite{Sp}. 
The geometric realisation of $X$ with the metric topology is denoted $|X|_d$. While for finite simplicial complexes the spaces $|X|$ and 
$|X|_d$ are homeomorphic, it is not true for infinite complexes in general. For the Rado complex $X$ the spaces $|X|$ and $|X|_d$ are not homeomorphic. Moreover,  in general, the metric topology is not invariant under subdivisions, see \cite{Mine}, where this issue is discussed in detail. We do not know if for the Rado complex $X$ the spaces $|X|_d$ and $|\Delta_\N|_d$ are homeomorphic. 
}
\end{remark}

\section{Infinite random simplicial complexes}

We show in the following \S \ref{sec6} that a random infinite simplicial complex is a Rado complex with probability 1, in a certain regime. 
In this section we prepare the grounds and describe the probability measure on the set of infinite simplicial complexes. 

\subsection{}
Let $L$ be a finite simplicial complex. We denote by $F(L)$ the set of simplexes of $L$; besides, $V(L)$ will denote the set of vertexes of $L$. Suppose that with each simplex $\sigma\subset L$ one has associated a probability parameter 
$p_\sigma\in [0,1]$. We shall use the notation $q_\sigma =1- p_\sigma$. 
Given a subcomplex $A\subset L$ we may consider the set $E(A|L)$ consisting of all simplexes of $L$ which are not in $A$ but such that 
all their proper faces are in $A$. Simplexes of $E(A|L)$ are called {\it external for $A$ in $L$}. As an example we mention that any vertex $v\in L-A$ is an external simplex, $v\in E(A|L)$. 

 With a subcomplex $A\subset L$ one may associate the following real number
\begin{eqnarray}\label{pa}
p(A) = \prod_{\sigma\in F( A)} p_\sigma \cdot \prod_{\sigma\in E(A|L)}q_\sigma \, \in \, [0,1].
\end{eqnarray}
For example, taking $A=\emptyset$ we obtain $p(\emptyset)=\prod_{v\in V(L)} q_v$, the product is taken with respect to all vertices $v$ of $L$. 

\begin{lemma} \label{lm21} One has
$\sum_{A\subset L} p(A) =1,$
where $A$ runs over all subcomplexes of $L$, including the empty subcomplex. 
\end{lemma}
The proof can be found in \S \ref{app}.

\subsection{} Let $\Delta=\Delta_\N$ denote the simplex spanned by the set $\N=\{ 1, 2, \dots\}$ of positive integers. 
We shall denote by $\Omega$ the set of all simplicial subcomplexes $X\subset \Delta$. Each simplicial complex $X\in \Omega$ has finite or countable set of vertexes $V(X)\subset \N$ and any finite or countable simplicial complex is isomorphic to one of the complexes $X\in \Omega$. 

\subsection{} Let $\Delta_n$ denote the simplex spanned by the vertices $[n]=\{1, 2, \dots, n\}\subset \N$. Let $\Omega_n$ denote the set of all subcomplexes $Y\subset \Delta_n$. One has the projection
$$\pi_n : \Omega \to \Omega_n, \quad X\mapsto X\cap \Delta_n.$$
In other words, for $X\in \Omega$ the complex $\pi_n(X)\subset \Delta_n$ is the subcomplex of $X$ induced on the vertex set $[n]\subset \N$. 

For a subcomplex $Y\subset \Delta_n$ we shall consider the set
\begin{eqnarray}
Z(Y, n) = \pi_n^{-1}(Y) = \{X\in \Omega; X\cap \Delta_n =Y\} \subset \Omega.
\end{eqnarray}
Note that for $n=n'$ the sets $Z(Y, n)$ and $Z(Y', n')$ are either identical (if and only if $Y=Y'$) of disjoint; for $n>n'$ the intersection 
$Z(Y, n)\cap Z(Y', n')$ is nonempty if and only if $Y\cap \Delta_{n'} = Y'$ and in this case $Z(Y, n) \subset Z(Y', n')$. 
Note also that for 
$n>n'$ and $Y\cap \Delta_{n'} = Y'$ one has 
\begin{eqnarray}\label{adds}
Z(Y', n') = \bigsqcup_j Z(Y_j, n)\end{eqnarray}
where $Y_j\subset \Delta_n$ are all subcomplexes with $Y_j\cap\Delta_{n'} =Y'$; one of these subcomplexes $Y_j$ coincides with $Y$.

Let $\mathcal A$ denote the set of all subsets $Z(Y, n)\subset \Omega$ and $\emptyset$. The set $\mathcal A$ is {\it a semi-ring}, see \cite{Kl}, i.e. 
$\mathcal A$ is $\cap$-closed and for any $A, B\in \mathcal A$ the difference $B-A$ is a finite union of mutually disjoint sets from $\mathcal A$. 
We shall denote by $\mathcal A'$ the $\sigma$-algebra generated by $\mathcal A$. 

\begin{example}\label{ex1}{\rm
Let $U\subset \N$ be a finite subset and let $L$ be a simplicial complex with vertex set $V(L)\subset U$. 
Then the set $\{X\in \Omega; X_U=L\}$ is the union of finitely many elements of the semi-ring $\mathcal A$ and in particular, 
$\{X\in \Omega; X_U=L\}\in \mathcal A'.$
Indeed, let $n$ be an integer such that $U\subset [n]$ and let $Y_j\subset \Delta_n$, for $j\in I$, be the list of all subcomplexes of $\Delta_n$ satisfying $(Y_j)_U=L$; in other words, $Y_j$ induces $L$ on $U$. Then the set $\{X\in \Omega; X_U=L\}$ is the union 
$\sqcup_{j\in I} Z(Y_j, n).$
}
\end{example}

\subsection{} Next we defile a function $\mu: \mathcal A \to \R$ as follows. Fix for every simplex $\sigma\subset \Delta_\N$ a probability parameter $p_\sigma\in [0,1]$. The function  
\begin{eqnarray}\label{system}
 F(\Delta_\N) \to [0, 1], \quad \sigma \mapsto \{p_\sigma\}
\end{eqnarray}
 will be called {\it the system of probability parameters. } Here $\sigma$ runs over all simplexes $\sigma \in F(\Delta_\N)$. We shall use the notation $q_\sigma =1-p_\sigma$. 

For an integer $n\ge 0$ and a subcomplex $Y\subset \Delta_n$ define 
\begin{eqnarray}\label{prod}
\mu(Z(Y, n)) = \prod_{\sigma\in F(Y)} p_\sigma \cdot \prod_{\sigma\in E(Y|\Delta_n)} q_\sigma.
\end{eqnarray}

Let us show that $\mu$ is additive. We know that the set $Z(Y,n)$ equals the disjoint union 
\begin{eqnarray}\label{disjoint}
Z(Y,n) \, =\, \sqcup_{j\in I} Z(Y_j, n+1)\end{eqnarray} 
where $Y_j$ are all subcomplexes of $\Delta_{n+1}$ satisfying $Y_j \cap \Delta_n =Y$. One of these subcomplexes $Y_{j_0}$ equals $Y$ and the others contain the vertex $(n+1)$ and have the form $$Y_j=Y\cup ((n+1)\ast A_j)$$ where $A_j\subset Y$ is a subcomplex. In other words, all complexes $Y_j$ with $j\not=j_0$ are obtained from $Y$ by adding a cone with apex $n+1$ over a subcomplex $A_j\subset Y$. Clearly, any subcomplex $A_j\subset Y$ may occur, including the empty subcomplex $A_j=\emptyset$. 

Applying the definition (\ref{prod}) we have
$$\mu(Z(Y, n+1) = \mu(Z(Y, n))\cdot q_{n+1},$$ and for $j\not=j_0$, 
\begin{eqnarray}
\mu(Z(Y_j, n+1) = \mu(Z(Y, n))\cdot p_{n+1} \cdot \prod_{\sigma\in F(A_j)} p'_\sigma \cdot \prod_{\sigma\in E(A_j|Y)} q'_\sigma,
\end{eqnarray}
where $n+1$ denotes the new added vertex and $p'_\sigma$ denotes the probability parameter $p_{(n+1)\sigma}$ associated to the simplex $(n+1)\ast \sigma$  (the cone over $\sigma$ with apex $n+1$); besides, $q'_\sigma= 1- p'_\sigma$. Hence we obtain, using Lemma \ref{lm21}:
\begin{eqnarray*}
&& \sum_{j\in I} \mu(Z(Y_j, n+1)) \\
&&= \mu(Z(Y, n))\cdot \left\{
q_{n+1} + p_{n+1} \cdot \left[\sum_{A_j\subset Y} \prod_{\sigma\in F(A_j)} p'_\sigma \cdot \prod_{\sigma\in E(A_j|Y} q'_\sigma\right]\right\}\\
&&= \mu(Z(Y, n)).
\end{eqnarray*} Thus we see that $\mu$ is additive with respect to relations of type (\ref{disjoint}). But obviously, by (\ref{adds}), these relations generate all additive relations in $\mathcal A$. This implies that $\mu$ is additive.

Note that $\Omega$ can be naturally viewed as the inverse limit of the finite sets $\Omega_n$, i.e. 
$\Omega = \underset \leftarrow\lim \Omega_n.$ Introducing the discrete topology on each $\Omega_n$ we obtain the inverse limit topology on $\Omega$ and with this topology $\Omega$ is compact and totally disconnected; it is homeomorphic to the Cantor set. The sets $Z(Y, n)\subset \Omega$ are open and closed in this topology, hence they are compact.

Next we  apply Theorem 1.53 from \cite{Kl} to show that $\mu$ extends to a probability measure on the $\sigma$-algebra $\mathcal A'$ generated by $\mathcal A$. This theorem requires for $\mu$ to be additive, $\sigma$-subadditive and $\sigma$-finite. By Theorem 1.36 from \cite{Kl}, $\sigma$-subadditivity is equivalent to $\sigma$-additivity. Recall that $\sigma$-additivity means that for $A=\sqcup_i A_i$ (disjoint union of countably many elements of $\mathcal A$) one has $\mu(A) =\sum_i \mu(A_i)$. In our case, since the sets 
$A_i\subset \Omega$ are open and closed and since $\Omega$ is compact, any representation $A=\sqcup_i A_i$ must be finite and hence $\sigma$-additivity of $\mu$ follows from additivity. 

 For fixed $n$ we have $\Omega=\sqcup Z(Y, n)$ where $Y$ runs over all subcomplexes of $\Delta_n$ (including $\emptyset$).
 Using additivity of $\mu$ and applying Lemma \ref{lm21},  we have 
$\mu(\Omega) = \sum_{Y\subset \Delta_n} \mu(Z(Y, n)) =1.$ This shows that $\mu$ is $\sigma$-finite and hence by Theorem 1.53 from 
\cite{Kl} $\mu$ extends to a probability measure on $\mathcal A'$. The extended measure on $\mathcal A'$ will be denoted by the same symbol $\mu$.

%
\begin{example}\label{ex2}{\rm
As in Example \ref{ex1}, let $U\subset \N$ be a finite subset and let $L$ be a simplicial complex with vertex set $V(L)\subset U$. 
Then 
\begin{eqnarray}\label{prod1}
\mu(\{X\in \Omega; X_U=L\}) = \prod_{\sigma\in F(L)} p_\sigma \cdot \prod_{\sigma\in E(L|\Delta_{U})} q_\sigma.
\end{eqnarray}
Here $\Delta_U$ denotes the simplex spanned by $U$. The proof is left to the reader as an exercise. 
}
\end{example}

\section{Random simplicial complex in the medial regime is Rado}\label{sec6}

In this section we prove that an infinite random simplicial complex in the medial regime is a Rado complex with probability one.\footnote{Finite simplicial complexes in the medial regime were studied in \cite{FM1}.}

\begin{definition}
We shall say that a system of probability parameters $p_\sigma$, see (\ref{system}), is {\it in the medial regime} if there exist  
$0<p_-<p_+<1$ such that the probability parameter $p_\sigma$ satisfies $p_\sigma \in [p_-, p_+]$ for any simplex $\sigma\in F(\Delta_\N)$. 
\end{definition}

In other words, in the medial regime the probability parameters $p_\sigma$ are uniformly bounded away from zero and one.

\begin{theorem}\label{thm26}
A random simplicial complex with countably many vertexes in the medial regime is a Rado complex, with probability one. 
\end{theorem}

\begin{proof} For a finite subset $U\subset \N$ and for a simplicial subcomplex $A\subset \Delta_U$ of the simplex $\Delta_U$ consider the set 
\begin{eqnarray}
\Omega^{U, L}= \{X \in \Omega; X_U=L\}.
\end{eqnarray}
This set belongs to the $\sigma$-algebra $\mathcal A'$ and has positive measure, see Example \ref{ex2}. 

Consider also the subset 
$\Omega^{U, L, A, v}\subset \Omega^{U, L}$
consisting of all subcomplexes $X\in \Omega$ satisfying $X_{U\cup v}=L\cup vA$. Here $A\subset L$ is a subcomplex and $v\in \N-U$. 

The conditional probability equals
$$\mu(\Omega^{U, L, A, v}|\Omega^{U, L}) = p_v\cdot \prod_{\sigma\in F(A)} p_{v\sigma} \cdot \prod_{\sigma\in E(A|L)} q_{v\sigma}\ge 
p_-^{|F(A)|}(1-p_+)^{|E(A|L)|}>0,$$
see (\ref{prod1}).
Note that the events $\Omega^{U, L, A, v}$, conditioned on $\Omega^{U, L}$ for various $v$, are independent and the sum of their probabilities is $\infty$. Hence we may apply the Borel-Cantelli Lemma (see \cite{Kl}, page 51) to conclude that the set of complexes $X\in \Omega^{U, L}$ such that 
$X_{U\cup v}=L\cup vA$ for infinitely many vertices $v$ has full measure in $\Omega^{U, L}$. 

By taking a finite intersection with respect to all possible subcomplexes $A\subset L$ this implies that the set 
$\Omega_\ast^{U, L}\subset \Omega^{U, L}$ of simplicial complexes $X\in \Omega^{U, L}$ such that for any subcomplex $A\subset L$ there exists infinitely many 
vertexes $v$ with $X_{U\cup v}=L\cup vA$ has full measure in $\Omega^{U, L}$. 

Since $\Omega = \cap_U \cup_{L\subset \Delta_U} \Omega^{U, L}$ (where $U\subset \N$ runs over all finites subsets) we obtain that the set 
$\cap_U \cup_{L\subset \Delta_U} \Omega^{U, L}_\ast$ has measure 1 in $\Omega$. But the latter set $\cap_U \cup_{L\subset \Delta_U} \Omega^{U, L}_\ast$ is exactly the set of all Rado simplicial complexes, see Lemma \ref{lm31}. 
\end{proof}

\section{Random induced subcomplexes of a Rado complex}

In this section we consider a different situation. Let $X$ be a fixed Rado complex with vertex set $V(X)=\N$. Suppose that each of the vertexes $n\in \N$ is selected at random with probability $p_n\in [0, 1]$ independently of the selection of all other vertexes. 
Denote by $X_\omega$ the subcomplex of $X$ induced on the selected set of vertexes. Here $\omega$ stands for the selection sequence, one may think that $\omega\in \{0, 1\}^\N$. Under which condition on the sequence $\{p_n\}$ the complex $X_\omega$ is Rado with probability 1?

Applying Borel-Cantelli Lemma we get: 

(1) If $\sum p_n <\infty$ then complex $X_\omega$ has finitely many vertexes, with probability 1. 

(2) If $\sum p_n =\infty$ then complex $X_\omega$ has infinitely many vertexes, with probability 1. 

(3) If $\sum q_n <\infty$ (where $q_n=1-p_n$) then the set of vertexes of $X_\omega$ has a finite complement in $\N$ and hence $X_\omega$ is a Rado complex with probability 1. In (3) we use Lemma \ref{lm16}. 

The following result strengthens point (3) above:

\begin{lemma}
Suppose that  for some $p>0$ one has $p_n\ge p>0$ for any $n\in \N$. Then $X_\omega$ is a Rado complex with probability 1. 
\end{lemma}
\begin{proof}
Denote by $X_n$ the subcomplex of $X$ induced on the set $[n]= \{1, 2, \dots, n\}\subset \N$. For a subcomplex $A\subset X_n$ consider the set of vertexes $$W(A, n)=\{v\in \N -[n]; \lk_X(v)\cap X_n =A\}.$$
We know that this set is infinite (see Lemma \ref{lm31}) and since each of the elements of this set is a vertex of $X_\omega$ with probability 
at least $p>0$ we obtain (using Borel - Cantelli) that with probability 1 the intersection 
$V(X_\omega) \cap W(A, n)$
is infinite. Hence the set 
\begin{eqnarray}\label{intersection}
\bigcap_{n=1}^\infty \bigcap_{A\subset X_n} \{\omega; |V(X_\omega)\cap W(A, n)|=\infty\} 
\end{eqnarray}
has measure 1 (as intersection of countably many sets of measure one). Here we use $\sigma$-additivity of the Bernoulli measure. 
It is obvious that for any $\omega$ lying in the intersection (\ref{intersection}) the induced complex $X_\omega$ is ample and hence Rado. 
\end{proof}

\section{Proof of Lemma \ref{lm21}}\label{app}
We obviously have
\begin{eqnarray}\label{sum1}
1=\prod_{\sigma\in F(L)}(p_\sigma +q_\sigma) = \sum_{J\subset F(L)} \left(\prod_{\sigma\in J} p_\sigma \cdot \prod_{\sigma\notin J}  q_\sigma\right).
\end{eqnarray}
Note that in the above sum, $J$ can be also the empty set. Denote by $A(J)\subset J$ the set of all simplexes $\sigma\in J$ such that for any face $\tau\subset \sigma$ one has $\tau\in J$. Note that $A= A(J)$ is a simplicial complex, it is the largest simplicial subcomplex of $L$
with $F(A)\subset J$. We also note that the set of external simplexes $E(A|L)$ is disjoint from $J$. 

Fix a subcomplex $A\subset L$ and consider all subsets $J\subset F(L)$ with $A(J)=A$. 
Any such subset $J\subset F(L)$ contains $F(A)$ and is disjoint from $E(A|L)$. 
Conversely, any subset $J\subset F(L)$ containing $F(A)$ and disjoint from $E(A|L)$ satisfies $A(J)=A$. 

Denoting $S(A) = F(L) - F(A) -E(A|L)$ and $I=J\cap S(A)$ we see that any term of (\ref{sum1}) corresponding to a subset $J$ with $A(J)=A$ 
can be written in the form 
\begin{eqnarray}
\left(\prod_{\sigma\in F(A)} p_\sigma \cdot \prod_{\sigma\in E(A|L)} q_\sigma \right) \cdot \left( \prod_{\sigma\in I} p_\sigma \cdot \prod_{\sigma\in S(A)-I} q_\sigma \right) 
\end{eqnarray}
and the first factor above is $p(A)$, see (\ref{pa}). Hence the sum of all terms in the sum (\ref{sum1}) corresponding to the subsets $J$ with $A(J)=A$ equals
\begin{eqnarray}
p(A) \cdot \sum_{I\subset S(A)} \left( \prod_{\sigma\in I} p_\sigma \cdot \prod_{\sigma\in S(A)-I} q_\sigma \right)= p(A)\cdot \prod_{\sigma\in S(A)} (p_\sigma +q_\sigma) = p(A).
\end{eqnarray}
We therefore see that the statement of Lemma \ref{lm21} follows from (\ref{sum1}).

\end{document}